\documentclass[psreqno,reqno,11pt]{amsart}
\usepackage{eurosym}
\usepackage{amssymb}
\usepackage{tikz}
\usepackage{amsmath}
\usepackage{tikz}
\usepackage{pgflibraryplotmarks}
\usepackage{wasysym}
\usepackage{stmaryrd}
\usepackage[english]{babel}

\usepackage{hyperref}

\theoremstyle{plain}
\newtheorem{theorem}{Theorem}[section]
\newtheorem{proposition}[theorem]{Proposition}
\newtheorem{lemma}[theorem]{Lemma}
\newtheorem{corollary}[theorem]{Corollary}

\theoremstyle{definition}
\newtheorem{definition}[theorem]{Definition}
\newtheorem{notation}[theorem]{Notation}

\allowdisplaybreaks

\begin{document}
	\def\N{\mathbb{N}}
	\def\Z{\mathbb{Z}}
	\def\R{\mathbb{R}}
	\def\C{\mathbb{C}}
	\def\M{\mathcal{M}}
	
	\title[Simpler characterizations of total orderization invariant maps]{
		Simpler characterizations of total orderization invariant maps}
	\author{C. Schwanke}
	\address{Department of Mathematics and Applied Mathematics, University of Pretoria, Private Bag X20, Hatfield 0028, South Africa}
	\email{cmschwanke26@gmail.com}
	\date{\today}
	
\begin{abstract}
Given a finite subset $A$ of a distributive lattice, its total orderization $to(A)$ is a natural transformation of $A$ into a totally ordered set. Recently, the author showed that multivariate maps on distributive lattices which remain invariant under total orderizations generalize various maps on vector lattices, including bounded orthosymmetric multilinear maps and finite sums of bounded orthogonally additive polynomials. Therefore, a study of total orderization invariant maps on distributive lattices provides new perspectives for maps widely researched in vector lattice theory. However, the unwieldy notation of total orderizations can make calculations extremely long and difficult. In this paper we resolve this complication by providing considerably simpler characterizations of total orderization maps. Utilizing these easier representations, we then prove that a lattice multi-homomorphism on a distributive lattice is total orderization invariant if and only if it is symmetric, and we show that the diagonal of a symmetric lattice multi-homomorphism is a lattice homomorphism, extending known results for orthosymmetric vector lattice homomorphisms.
\end{abstract}
	
	\maketitle
	\section{Introduction}\label{S:intro}

Several characterizations of bounded orthosymmetric multilinear maps and orthogonally additive polynomials have been studied in recent years, see e.g. \cite{BenAmor, BoyRySniga, BusSch-characterizing, Kus3, Kusa, Kusa-on the rep of orth adds, Sch, Sch-toi}. Combining \cite[respectively, Theorem~18, Lemma~2.6, Theorems~2.3\&2.4, main~result, Lemma~4, Theorem~2.3, Theorems~3.14\&3.17]{BenAmor, BoyRySniga, BusSch-characterizing, Kusa, Kusa-on the rep of orth adds, Sch, Sch-toi} all into one theorem, we obtain the following result.

\begin{theorem}\cite{BenAmor, BoyRySniga,BusSch-characterizing, Kusa, Kusa-on the rep of orth adds, Sch, Sch-toi}\label{T:synth}
Let $E$ be a uniformly complete Archimedean vector lattice, let $Y$ be a real separated bornological space, put $r,n\in\mathbb{N}\setminus\{1\}$, let $T\colon E^n\to Y$ be a bounded symmetric $n$-linear map, and let $P_T$ be the $n$-homogeneous polynomial generated by $T$. The following are equivalent.
\begin{itemize}
	\item[(i)] $T$ is orthosymmetric,
	\item[(ii)] $T$ is total orderization invariant,
	\item[(iii)] $T$ is positively total orderization invariant,
	\item[(iv)] $T(\underbrace{x, \dots, x}_{n-k\ \text{times}}, \underbrace{y, \dots, y}_{k\ \text{times}})=0$ for every $k\in\lbrace 1,\dots,n-1\rbrace$ and all $x,y\in E$ with $x\perp y$,
	\item[(v)] $P_T$ is orthogonally additive,
	\item[(vi)] $P_T$ is positively orthogonally additive,
	\item[(vii)] $P_T(x)=P_T(x^+)+(-1)^nP_T(x^-)$ for each $x\in E$,
	\item[(viii)] $P_T\bigl(\mathfrak{S}_{n}(x_1,\dots,x_r)\bigr)=\sum_{k=1}^{r}P_T(x_k)$ holds for all $x_{1},\dots,x_{r}\in E^{+}$, where $\mathfrak{S}_n$ denotes the $n$th root mean power, defined via the Archimedean vector lattice functional calculus,
	\item[(ix)] $P_T\bigl(\mathfrak{G}_n(x_1,\dots,x_n)\bigr)=T(x_1,\dots,x_n)$ holds for all $x_{1},\dots,x_{n}\in E^{+}$, where $\mathfrak{G}_n$ denotes the $n$-fold geometric mean, defined by the Archimedean vector lattice functional calculus,
	\item[(x)] $P_T(|z|)=(P_T)_{\C}(z^{\frac{n}{2}}(\bar{z})^{\frac{n}{2}})$ holds for all $z\in E_{\C}$ if $n$ is even, while if $n$ is odd, then $P_T(|z|)=(P_T)_{\C}(z^{\frac{n-1}{2}}(\bar{z})^{\frac{n-1}{2}}|z|)$ holds for every $z\in E_{\C}$,
		\item[(xi)] The map $(x_1,...,x_n)\mapsto\sum_{k=1}^{n}P_T(x_k)$ is total orderization invariant, and
	\item[(xii)] The map $(x_1,..,x_n)\mapsto\sum_{k=1}^{n}P_T(x_k)$ is positively total orderization invariant,
\end{itemize}
\end{theorem}

Out of all the characterizations given above, $(ii)$ alone stands out as unique in that it is the only item that still makes sense if the vector lattice $E$ in Theorem~\ref{T:synth} is replaced by a more general distributive lattice (which may not possess a vector space structure), and if $T$ is not necessarily multilinear. From this point of view, total orderization invariant maps on distributive lattices are generalizations of bounded orthosymmetric multilinear maps and finite sums of bounded orthogonally additive polynomials on uniformly complete Archimedean vector lattices.

This paper illustrates how total orderization invariant maps are intriguing in their own right, even if they are not compatible with a vector space. Our study of these maps on distributive lattices in turn reveals new information regarding orthosymmetric and orthogonally additive maps on vector lattices.

As their name suggests, the distinguishing feature of total orderization invariant maps is that they remain unchanged under the total orderization, which is perhaps the most natural manner in which to transform a finite subset of a distributive lattice into a totally ordered set (see \cite[Section~2]{Sch-toi} and Section~\ref{S:Prelims} of this document below for more details). Other desirable attributes of these maps include the fact that they are symmetric \cite[Proposition~2.9]{Sch-toi} and that they are completely determined by their behavior on totally ordered sets. For example, if $L$ is a distributive lattice, $A$ is a nonempty set, $T\colon L^n\to A$ is totally orderization invariant, $B\subseteq A$, and $T(x_1,...,x_n)\in B$ whenever $x_1,...,x_n\in L$ and $x_1\leq x_2\leq \cdots \leq x_n$, then $T(x_1,...,x_n)\in B$ for all $x_1,...,x_n\in L$.

Despite these interesting properties of total orderization maps, one can see from \cite{Sch-toi} and Section~\ref{S:Prelims} below that total orderizations involve cumbersome expressions. Therefore, determining whether or not a map is total orderization invariant, for example, can be especially laborious. In the spirit of Theorem~\ref{T:synth}, however, we provide in Section~\ref{S:CTOIM} new characterizations of total orderization invariant maps on distributive lattices that are significantly more manageable.

In Section~\ref{S:TOIMH} we then illustrate the utility of these simpler descriptions of totally orderization invariant maps by proving in Theorem~\ref{T:charoftoimultihoms} that a lattice multi-homomorphism between distributive lattices is total orderization invariant if and only if it is symmetric. This result generalizes \cite[Lemma~2.1]{BoBus}, which states that a vector lattice multimorphism between Archimedean vector lattices is symmetric if and only if it is orthosymmetric.

A related result of interest in this context is the following theorem of Kusraev.

\begin{theorem}\label{T:Kus}\cite[Theorem~1]{Kus3}
Let $E$ and $F$ be uniformly complete Archimedean vector lattices, and let $b\colon E\times E\to F$ be a vector lattice bimorphism. Then the following are equivalent.

\begin{itemize}
	\item[(i)] $b$ is symmetric,
	\item[(ii)] $b(x,x)-b(y,y)=b(x-y, x+y)$ for all $x,y\in E$,
		\item[(iii)] $b(x,x)\wedge b(y,y)\leq b(x,y)\leq b(x,x)\vee b(y,y)$ for all $x,y\in E^+$ 
			\item[(iv)] $b(x\wedge y, x\wedge y)=b(x,x)\wedge b(y,y)$ and $b(x\vee y, x\vee y)=b(x,x)\vee b(y,y)$ for all $x,y\in E^+$,
				\item[(v)] $b(x,y) = b(y,x)$ for all disjoint $x,y\in E$,
					\item[(vi)] $b(x,|x|)=b(x^+, x^+)-b(x^-, x^-)$,
						\item[(vii)] $b$ is orthosymmetric, and
							\item[(viii)] $b$ is positively semidefinite. 
\end{itemize}
\end{theorem}

Glancing at this result, one sees that items $(i), (iii)$, and $(iv)$ all make sense even if $b$ was not necessarily bilinear and if $E$ in $(i)$, and $E^+$ in $(iii)$ and $(iv)$, were replaced with a more general distributive lattice. Given that vector lattice multimorphisms are orthosymmetric if and only if they are total orderization invariant (vector lattice multimorphisms are order bounded), one can also generalize $(vii)$ to the more general setting of lattice bi-homomorphisms $b$ on distributive lattices that are total orderization invariant. 

The equivalence of all these assertions no longer holds in the more general distributive lattice setting. For a simple counterexample, consider the following weighted geometric mean $b\colon \mathbb{R}^+\times\mathbb{R}^+\to \mathbb{R}^+$ defined by
\[
b(x,y):=x^{1/3}y^{2/3}.
\]
Then $b$ is a lattice bi-homomorphism (though not bilinear) and $b(z,z)=z$ for all $z\in \mathbb{R}^+$. Thus $b$ satisfies the identities in $(iv)$ above. However, $b$ is not symmetric; i.e. does not satisfy item $(i)$ in Theorem~\ref{T:Kus}.

However, as mentioned previously, we show in Theorem~\ref{T:charoftoimultihoms} that conditions $(i)$ and $(vii)$ remain equivalent in our more general setting and for several variables. Furthermore, in this setting, and for any finite number of arguments, the equivalence of $(iii)$ and $(iv)$ for symmetric lattice multi-homomorphisms and the implication $(i)/(vii)\implies(iii)/(iv)$ are also proven in Lemma~\ref{L:kusraev} and Theorem~\ref{T:diags}, respectively.

We proceed with some preliminary material.

\section{Preliminaries}\label{S:Prelims}

We refer the reader to \cite{AB, Birk, LuxZan1, Zan2}  for any unexplained terminology or basic theory regarding distributive lattices and vector lattices.
Throughout this paper, $\mathbb{N}$ stands for the set of strictly positive integers, and the ordered field of real numbers is denoted by $\mathbb{R}$. Given a set $A$ and $n\in\N$, we as usual denote the $n$-fold Cartesian product of $A$ by $A^n$.

\begin{notation}
We use the convenient notation $[n]:=\{1,...,n\}$ throughout the rest of this paper.
\end{notation}

Let $L$ and $M$ be lattices. A map $T\colon L\to M$ is called a \textit{lattice homomorphism} if $T(x\vee y)= T(x)\vee T(y)$ and $T(x\wedge y)=T(x)\wedge T(y)$ for every $x,y\in L$. More generally, for $n\in\N$, we call a map $T\colon L^n\to M$ a \textit{lattice $n$-homomorphism} if
\[
T(x_1,..., x_{i-1}, x_i\vee y, x_{i+1}, ..., x_n)=T(x_1,..., x_n)\vee T(x_1,..., x_{i-1}, y, x_{i+1}, ..., x_n)
\]
and
\[
T(x_1,..., x_{i-1}, x_i\wedge y, x_{i+1}, ..., x_n)=T(x_1,..., x_n)\wedge T(x_1,..., x_{i-1}, y, x_{i+1}, ..., x_n)
\]
both hold for all $x_1,....,x_n, y\in L$ and every $i\in\{1,...,n\}$. If $T$ is a lattice $n$-homomorphism for some $n\in\N$, we call $T$ a \textit{lattice multi-homomorphism}.

We next recall a definition given in  \cite{Sch-toi}, which is central to the definition of total orderization invariant maps.

\begin{definition}
Let $L$ be a distributive lattice. For $n\in\mathbb{N}$, $x_1,...,x_n\in L$, and $k\in[n]$, we define the \textit{$k$th median} of $\{x_1,...,x_n\}$ to be
\[
\mathcal{M}_k(x_1,...,x_n):=\underset{\substack{i_1,i_2,...,i_{n+1-k}\in[n] \\ i_1<i_2<\cdots<i_{n+1-k}}}{\bigvee}\left(\bigwedge_{j=1}^{n+1-k}x_{i_j}\right)=\underset{\substack{i_1,i_2,...,i_{n+1-k}\in[n] \\ i_1<i_2<\cdots<i_{n+1-k}}}{\bigwedge}\left(\bigvee_{j=1}^{n+1-k}x_{i_j}\right).
\]
\end{definition}

One can see that the first median of $\{x_1,...,x_n\}$ is $\bigwedge_{i=1}^n x_i$ and that the $n$th median of this set is $\bigvee_{i=1}^n x_i$. More generally, the $k$th median is an abstraction of the notion of $k$th smallest element in a totally ordered set, as the following proposition states.

\begin{proposition}\cite[Proposition~2.5]{Sch-toi}\label{P:3ton}
	Let $n\in\mathbb{N}$ and $L$ be a distributive lattice. Then
	\begin{itemize}
		\item[$(i)$] $\mathcal{M}_k\colon L^n\to L$ is a symmetric function for all $k\in[n]$,
		\item[$(ii)$] if $x_1,x_2,...,x_n\in L$ satisfy $x_1\leq x_2\leq\cdots\leq x_n$, then $\mathcal{M}_k(x_1,...,x_n)=x_k$ for each $k\in[n]$, and
		\item[$(iii)$] $\mathcal{M}_1(x_1,...,x_n)\leq\mathcal{M}_2(x_1,...,x_n)\leq\cdots\leq\mathcal{M}_n(x_1,...,x_n)$ for every $x_1,x_2,...,x_n\in L$.
	\end{itemize}
\end{proposition}

Proposition~\ref{P:3ton}$(iii)$ above suggests the following definition, which is introduced in \cite{Sch-toi}.

\begin{definition}
	Given a distributive lattice $L$, $n\in\mathbb{N}$, and $x_1,...,x_n\in L$, we write
	\[
	to(\{x_1,...,x_n\}):=\{\mathcal{M}_1(x_1,...,x_n),\mathcal{M}_2(x_1,...,x_n),\dots,\mathcal{M}_n(x_1,...,x_n)\}
	\]
	and call the totally ordered set $to(\{x_1,...,x_n\})$ the \textit{total orderization} of $\{x_1,...,x_n\}$.

We add here that for $(x_1,...,x_n)\in L^n$, we define the \textit{total orderization} of $(x_1,...,x_n)$ in a similar manner:
\[
to(x_1,...,x_n):=\Bigl(\mathcal{M}_1(x_1,...,x_n),\mathcal{M}_2(x_1,...,x_n),\dots,\mathcal{M}_n(x_1,...,x_n)\Bigr).
\]
\end{definition}

Next we formally define the maps which are the central focus of this paper.

\begin{definition}
	Let $n\in\mathbb{N}$, assume $L$ is a distributive lattice, and suppose $A$ is a nonempty set. A map $T\colon L^n\to A$ is said to be \textit{total orderization invariant} if
	\[
	T(x_1,...,x_n)=T\Bigl(\mathcal{M}_1(x_1,...,x_n),\mathcal{M}_2(x_1,...,x_n),\dots,\mathcal{M}_n(x_1,...,x_n)\Bigr)
	\]
	holds for every $x_1,...,x_n\in L$.
\end{definition}

Examples of total orderization invariant maps include any symmetric function defined on a totally ordered set, the $k$th median maps on distributive lattices, addition in vector lattices, $f$-algebra multiplication on Archimedean semiprime $f$-algebras, positively homogeneous functions defined on Archimedean vector lattices via functional calculus,  bounded orthosymmetric multilinear maps on uniformly complete Archimedean vector lattices with separated real bornological space as codomain, and finite sums of bounded orthogonally additive polynomials in the same setting.
	
\section{Characterizations of Total Orderization Invariant Maps}\label{S:CTOIM}

We prove the main theorem of this paper (Theorem~\ref{T:charsoftoi}), which provides characterizations of total orderization invariant maps on distributive lattices which are substantially easier to work with than the definition given in \cite{Sch-toi}. We first require two lemmas. Lemma~\ref{L:Mk'saretoi} is a special case of Theorem~\ref{T:charsoftoi}.

\begin{lemma}\label{L:Mk'saretoi} Let $L$ be a distributive lattice and $k,n\in\mathbb{N}$ with $k\leq n$. Then for all $x_1,...,x_n\in L$, we have
\[
\M_k(x_1,...,x_n)=\M_k(x_1\wedge x_2, x_1\vee x_2, x_3,...,x_n).
\]
\end{lemma}

\begin{proof}
Let $x_1,...,x_n\in L$. First note that for $k=1$, the result is evident. For $k=2$ we have
\begin{align*}
\M_2&(x_1,...,x_n)=\underset{\begin{subarray}{c}
		i_1,\dots, i_{n-1}\in [n] \\
		i_1<\cdots<i_{n-1}
\end{subarray}}{\bigvee}\left(\bigwedge_{j=1}^{n-1}x_{i_j}\right)\\
&=(x_1\wedge\cdots\wedge x_{n-1})\vee (x_1\wedge\cdots\wedge x_{n-2}\wedge x_n)\vee(x_1\wedge\cdots\wedge x_{n-3}\wedge x_{n-1}\wedge x_n)\\
&\qquad\qquad\vee\cdots\vee(x_2\wedge\cdots\wedge x_n).
\end{align*}
In the same manner, and using the above string of equalities in the last equality below, we have
\begin{align*}
&\M_2(x_1\wedge x_2, x_1\vee x_2, x_3,...,x_n)\\
	\\
	&=\Bigl((x_1\wedge x_2)\wedge (x_1\vee x_2)\wedge x_3\wedge\cdots\wedge x_{n-1}\Bigr)\vee \Bigl((x_1\wedge x_2)\wedge (x_1\vee x_2)\wedge x_3\wedge\cdots\wedge x_{n-2}\wedge x_n\Bigr)\\
	&\qquad\qquad\vee\Bigl((x_1\wedge x_2)\wedge (x_1\vee x_2)\wedge x_3\wedge\cdots\wedge x_{n-3}\wedge x_{n-1}\wedge x_n\Bigr)\\
	&\qquad\qquad\vee\cdots\vee\Bigl((x_1\wedge x_2)\vee (x_1\vee x_2)\wedge x_4\wedge\cdots\wedge x_n\Bigr)\\
	&\qquad\qquad\vee \Bigl((x_1\wedge x_2)\wedge x_3\wedge\cdots\wedge x_n\Bigr)\vee \Bigl((x_1\vee x_2)\wedge x_3\wedge\cdots\wedge x_n\Bigr)\\
	\\
	\\
	\\
	&=(x_1\wedge x_2\wedge x_3\wedge\cdots\wedge x_{n-1})\vee (x_1\wedge x_2\wedge x_3\wedge\cdots\wedge x_{n-2}\wedge x_n)\\
	&\qquad\qquad\vee(x_1\wedge x_2\wedge x_3\wedge\cdots\wedge x_{n-3}\wedge x_{n-1}\wedge x_n)\\
	&\qquad\qquad\vee\cdots\vee (x_1\wedge x_2\wedge x_4\wedge\cdots\wedge x_n)\\
	&\qquad\qquad\vee \Bigl(x_1\wedge x_2\wedge x_3\wedge\cdots\wedge x_n\Bigr)\vee \Bigl((x_1\vee x_2)\wedge x_3\wedge\cdots\wedge x_n\Bigr)\\
	\\
	&=(x_1\wedge x_2\wedge x_3\wedge\cdots\wedge x_{n-1})\vee (x_1\wedge x_2\wedge x_3\wedge\cdots\wedge x_{n-2}\wedge x_n)\\
	&\qquad\qquad\vee(x_1\wedge x_2\wedge x_3\wedge\cdots\wedge x_{n-3}\wedge x_{n-1}\wedge x_n)\\
	&\qquad\qquad\vee\cdots\vee (x_1\wedge x_2\wedge x_4\wedge\cdots\wedge x_n) \vee \Bigl((x_1\vee x_2)\wedge x_3\wedge\cdots\wedge x_n\Bigr)\\
	\\
&=(x_1\wedge x_2\wedge x_3\wedge\cdots\wedge x_{n-1})\vee (x_1\wedge x_2\wedge x_3\wedge\cdots\wedge x_{n-2}\wedge x_n)\\
&\qquad\qquad\vee(x_1\wedge x_2\wedge x_3\wedge\cdots\wedge x_{n-3}\wedge x_{n-1}\wedge x_n)\\
&\qquad\qquad\vee\cdots\vee (x_1\wedge x_2\wedge x_4\wedge\cdots\wedge x_n)\vee (x_1\wedge x_3\wedge\cdots\wedge x_n)\vee (x_2\wedge x_3\wedge\cdots\wedge x_n)\\
\\
&=\M_2(x_1,...,x_n).
\end{align*}

Next we assume that $3\leq k\leq n$ and note that this supposition is needed in order to avoid empty suprema and infima in the string of equalities below. Set $y_1:=x_1\wedge x_2$, $y_2:=x_1\vee x_2$, and $y_i:=x_i$ for all $i\in\{3,...,n\}$. We need to show that
\[
\M_k(y_1, ..., y_n)=\M_k(x_1, ..., x_n).
\]

Using basic manipulation of suprema and infima, including distributivity, we obtain
\begin{align*}
&\M_k(y_1,...,y_n)=\underset{\begin{subarray}{c}
		i_1,\dots, i_{n+1-k}\in [n] \\
		i_1<\cdots<i_{n+1-k}
\end{subarray}}{\bigvee}\left(\bigwedge_{j=1}^{n+1-k}y_{i_j}\right)\\
\\
\\
\\
\\
&=\underset{\begin{subarray}{c}
		3\leq i_1,\dots, i_{n+1-k}\leq n \\
		i_1<\cdots<i_{n+1-k}
\end{subarray}}{\bigvee}\left(\bigwedge_{j=1}^{n+1-k}y_{i_j}\right)\vee\underset{\begin{subarray}{c}
i_1=1\\
3\leq i_2,\dots, i_{n+1-k}\leq n \\
i_2<\cdots<i_{n+1-k}
\end{subarray}}{\bigvee}\left(\bigwedge_{j=1}^{n+1-k}y_{i_j}\right)\\
&\qquad\qquad\vee\underset{\begin{subarray}{c}
i_1=2\\
3\leq i_2,\dots, i_{n+1-k}\leq n \\
i_2<\cdots<i_{n+1-k}
\end{subarray}}{\bigvee}\left(\bigwedge_{j=1}^{n+1-k}y_{i_j}\right)\vee\underset{\begin{subarray}{c}
i_1=1\\
i_2=2\\
3\leq i_3,\dots, i_{n+1-k}\leq n \\
i_3<\cdots<i_{n+1-k}
\end{subarray}}{\bigvee}\left(\bigwedge_{j=1}^{n+1-k}y_{i_j}\right)\\
\\
&=\underset{\begin{subarray}{c}
		3\leq i_1,\dots, i_{n+1-k}\leq n \\
		i_1<\cdots<i_{n+1-k}
\end{subarray}}{\bigvee}\left(\bigwedge_{j=1}^{n+1-k}x_{i_j}\right)\vee\underset{\begin{subarray}{c}
		3\leq i_2,\dots, i_{n+1-k}\leq n \\
		i_2<\cdots<i_{n+1-k}
\end{subarray}}{\bigvee}\left[\left(\bigwedge_{j=2}^{n+1-k}x_{i_j}\right)\wedge(x_1\wedge x_2)\right]\\
&\qquad\qquad\vee\underset{\begin{subarray}{c}
		3\leq i_2,\dots, i_{n+1-k}\leq n \\
		i_2<\cdots<i_{n+1-k}
\end{subarray}}{\bigvee}\left[\left(\bigwedge_{j=2}^{n+1-k}x_{i_j}\right)\wedge(x_1\vee x_2)\right]\\
&\qquad\qquad\vee\underset{\begin{subarray}{c}
		3\leq i_3,\dots, i_{n+1-k}\leq n \\
		i_3<\cdots<i_{n+1-k}
\end{subarray}}{\bigvee}\left[\left(\bigwedge_{j=3}^{n+1-k}x_{i_j}\right)\wedge(x_1\wedge x_2)\wedge(x_1\vee x_2)\right]\\
\\
&=\underset{\begin{subarray}{c}
		3\leq i_1,\dots, i_{n+1-k}\leq n \\
		i_1<\cdots<i_{n+1-k}
\end{subarray}}{\bigvee}\left(\bigwedge_{j=1}^{n+1-k}x_{i_j}\right)\vee\underset{\begin{subarray}{c}
		3\leq i_2, \dots, i_{n+1-k}\leq n \\
		i_2<\cdots<i_{n+1-k}
\end{subarray}}{\bigvee}\left[\left(\bigwedge_{j=2}^{n+1-k}x_{i_j}\right)\wedge(x_1\vee x_2)\right]\\
&\qquad\qquad\vee\underset{\begin{subarray}{c}
		3\leq i_3,\dots, i_{n+1-k}\leq n \\
		i_3<\cdots<i_{n+1-k}
\end{subarray}}{\bigvee}\left[\left(\bigwedge_{j=3}^{n+1-k}x_{i_j}\right)\wedge(x_1\wedge x_2)\right]\\
&=\underset{\begin{subarray}{c}
		3\leq i_1,\dots, i_{n+1-k}\leq n \\
		i_1<\cdots<i_{n+1-k}
\end{subarray}}{\bigvee}\left(\bigwedge_{j=1}^{n+1-k}x_{i_j}\right)\\
&\qquad\qquad\vee\underset{\begin{subarray}{c}
		3\leq i_2, \dots, i_{n+1-k}\leq n \\
		i_2<\cdots<i_{n+1-k}
\end{subarray}}{\bigvee}\left[\left\{\left(\bigwedge_{j=2}^{n+1-k}x_{i_j}\right)\wedge x_1\right\} \vee \left\{\left(\bigwedge_{j=2}^{n+1-k}x_{i_j}\right)\wedge x_2\right\}\right]\\
&\qquad\qquad\vee\underset{\begin{subarray}{c}
		3\leq i_3,\dots, i_{n+1-k}\leq n \\
		i_3<\cdots<i_{n+1-k}
\end{subarray}}{\bigvee}\left[\left(\bigwedge_{j=3}^{n+1-k}x_{i_j}\right)\wedge(x_1\wedge x_2)\right]\\
\\
&=\underset{\begin{subarray}{c}
		3\leq i_1,\dots, i_{n+1-k}\leq n \\
		i_1<\cdots<i_{n+1-k}
\end{subarray}}{\bigvee}\left(\bigwedge_{j=1}^{n+1-k}x_{i_j}\right)\vee\underset{\begin{subarray}{c}
		3\leq i_2, \dots, i_{n+1-k}\leq n \\
		i_2<\cdots<i_{n+1-k}
\end{subarray}}{\bigvee}\left\{\left(\bigwedge_{j=2}^{n+1-k}x_{i_j}\right)\wedge x_1\right\}\\
&\qquad\qquad\vee \underset{\begin{subarray}{c}
3\leq i_2, \dots, i_{n+1-k}\leq n \\
i_2<\cdots<i_{n+1-k}
\end{subarray}}{\bigvee}\left\{\left(\bigwedge_{j=2}^{n+1-k}x_{i_j}\right)\wedge x_2\right\}\\
&\qquad\qquad\vee\underset{\begin{subarray}{c}
		3\leq i_3,\dots, i_{n+1-k}\leq n \\
		i_3<\cdots<i_{n+1-k}
\end{subarray}}{\bigvee}\left[\left(\bigwedge_{j=3}^{n+1-k}x_{i_j}\right)\wedge(x_1\wedge x_2)\right]\\
\\
&=\underset{\begin{subarray}{c}
		i_1,\dots, i_{n+1-k}\in [n] \\
		i_1<\cdots<i_{n+1-k}\\
		i_j\geq 3\ \forall\ j\in[n+1-k]
\end{subarray}}{\bigvee}\left(\bigwedge_{j=1}^{n+1-k}x_{i_j}\right)\vee\underset{\begin{subarray}{c}
		i_1,\dots, i_{n+1-k}\in [n] \\
		i_1<\cdots<i_{n+1-k}\\
		i_1=1\\
		i_j\geq 3\ \forall\ j\in[n+1-k]\setminus\{1\}
\end{subarray}}{\bigvee}\left(\bigwedge_{j=1}^{n+1-k}x_{i_j}\right)\\
&\qquad\qquad\vee \underset{\begin{subarray}{c}
		i_1,\dots, i_{n+1-k}\in [n] \\
		i_1<\cdots<i_{n+1-k}\\
		i_1=2\\
		i_j\geq 3\ \forall\ j\in[n+1-k]\setminus\{1\}
\end{subarray}}{\bigvee}\left(\bigwedge_{j=1}^{n+1-k}x_{i_j}\right)\vee\underset{\begin{subarray}{c}
		i_1,\dots, i_{n+1-k}\in [n] \\
		i_1<\cdots<i_{n+1-k}\\
		i_1=1\\
		i_2=2\\
		i_j\geq 3\ \forall\ j\in[n+1-k]\setminus\{1,2\}
\end{subarray}}{\bigvee}\left(\bigwedge_{j=1}^{n+1-k}x_{i_j}\right)\\
\\
&=\underset{\begin{subarray}{c}
		i_1,\dots, i_{n+1-k}\in [n] \\
		i_1<\cdots<i_{n+1-k}
\end{subarray}}{\bigvee}\left(\bigwedge_{j=1}^{n+1-k}x_{i_j}\right)\\
\\
&=\M_k(x_1,...,x_n).
\end{align*} 
\end{proof}

We will utilize the following notation in our next lemma, as well as in Theorem~\ref{T:charsoftoi}. It proves useful when considering $k$th medians relative to a subset of a given set of elements in a distributive lattice.

\begin{notation}
Let $L$ be a distributive lattice, put $k,m,n\in\mathbb{N}$ with $k\leq m\leq n$, and let $x_1,...,x_n\in L$. If we are interested in the $k$th median relative to the subset $\{x_1,...,x_m\}$, and not the entire set $\{x_1,...,x_n\}$, then to emphasize this subtle difference, we write
\[
\M_{k,m}(x_1,...,x_m):=\M_k(x_1,...,x_m).
\]
\end{notation}

Lemma~\ref{L:nton+1} below is an essential ingredient to Theorem~\ref{T:charsoftoi} and will be utilized repeatedly in its proof.

\begin{lemma}\label{L:nton+1} Let $L$ be a distributive lattice and $k,m\in\mathbb{N}$ with $k\leq m$. Then
\[
\M_{k,m}(x_1,...,x_m)\wedge\Bigl[\M_{k-1,m}(x_1,...,x_m)\vee x_{m+1}\Bigr]=\M_{k,m+1}(x_1,...,x_{m+1})	
\]
holds for every $x_1,...,x_{m+1}\in L$.
\end{lemma}
	
\begin{proof}
Let $x_1,...,x_{m+1}\in L$. Using distributivity, Proposition~\ref{P:3ton}(iii), the definition of $k$th medians, and basic algebra with suprema and infima, we obtain
\begin{align*}
&\M_{k,m}(x_1,...,x_m)\wedge\Bigl[\M_{k-1,m}(x_1,...,x_m)\vee x_{m+1}\Bigr]\\
&=\M_{k-1,m}(x_1,...,x_m)\vee\Bigl[\M_{k,m}(x_1,...,x_m)\wedge x_{m+1}\Bigr]\\
&=\M_{k-1,m}(x_1,...,x_m)\vee\left[\left\{\underset{\begin{subarray}{c}
		i_1,\dots, i_{m+1-k}\in [m] \\
		i_1<\cdots<i_{m+1-k}
\end{subarray}}{\bigvee}\left(\bigwedge_{j=1}^{m+1-k}x_{i_j}\right)\right\}\wedge x_{m+1}\right]\\
&=\M_{k-1,m}(x_1,...,x_m)\vee\left[\underset{\begin{subarray}{c}
		i_1,\dots, i_{m+1-k}\in [m] \\
		i_1<\cdots<i_{m+1-k}
\end{subarray}}{\bigvee}\left\{\left(\bigwedge_{j=1}^{m+1-k}x_{i_j}\right)\wedge x_{m+1}\right\}\right]\\
&=\left[\underset{\begin{subarray}{c}
		i_1,\dots, i_{m+1-(k-1)}\in [m] \\
		i_1<\cdots<i_{m+1-(k-1)}
\end{subarray}}{\bigvee}\left(\bigwedge_{j=1}^{m+1-(k-1)}x_{i_j}\right)\right]\vee\left[\underset{\begin{subarray}{c}
		i_1,\dots, i_{m+1-k}\in [m] \\
		i_1<\cdots<i_{m+1-k}
\end{subarray}}{\bigvee}\left\{\left(\bigwedge_{j=1}^{m+1-k}x_{i_j}\right)\wedge x_{m+1}\right\}\right]\\
&=\left[\underset{\begin{subarray}{c}
		i_1,\dots, i_{m+2-k}\in [m] \\
		i_1<\cdots<i_{m+2-k}
\end{subarray}}{\bigvee}\left(\bigwedge_{j=1}^{m+2-k}x_{i_j}\right)\right]\vee\left[\underset{\begin{subarray}{c}
		i_1,\dots, i_{m+1-k}\in [m] \\
		i_1<\cdots<i_{m+1-k}
\end{subarray}}{\bigvee}\left\{\left(\bigwedge_{j=1}^{m+1-k}x_{i_j}\right)\wedge x_{m+1}\right\}\right]\\
&=\left[\underset{\begin{subarray}{c}
		i_1,\dots, i_{m+2-k}\in [m+1] \\
		i_1<\cdots<i_{m+2-k}\\
		i_{m+2-k}\ \neq\ m+1
\end{subarray}}{\bigvee}\left(\bigwedge_{j=1}^{m+2-k}x_{i_j}\right)\right]\vee\left[\underset{\begin{subarray}{c}
		i_1,\dots, i_{m+1-k}\in [m+1] \\
		i_1<\cdots<i_{m+2-k}\\
		i_{m+2-k}\ =\ m+1
\end{subarray}}{\bigvee}\left(\bigwedge_{j=1}^{m+2-k}x_{i_j}\right)\right]\\
&=\underset{\begin{subarray}{c}
		i_1,\dots, i_{m+2-k}\in [m+1] \\
		i_1<\cdots<i_{m+2-k}
\end{subarray}}{\bigvee}\left(\bigwedge_{j=1}^{m+2-k}x_{i_j}\right)\\
&=\underset{\begin{subarray}{c}
		i_1,\dots, i_{(m+1)+1-k}\in [m+1] \\
		i_1<\cdots<i_{(m+1)+1-k}
\end{subarray}}{\bigvee}\left(\bigwedge_{j=1}^{(m+1)+1-k}x_{i_j}\right)\\
&=\M_{k,m+1}(x_1,...,x_{m+1}).
\end{align*}
\end{proof}

With Lemmas~\ref{L:Mk'saretoi}\&\ref{L:nton+1} established, we are ready to present and prove the primary theorem of this paper. We will make use of the following notation in its proof.

\begin{notation}
Let $L$ be a distributive lattice. For $k, p\in \mathbb{N}$ with $k\leq p$ and $x_1,...,x_p\in L$ we will sometimes for short write
\[
\M_{k,p}(\bar{x}_{(p)}):=\M_{k,p}(x_1,...,x_p).
\]
\end{notation}

Theorem~\ref{T:charsoftoi} provides simpler characterizations of total orderization invariant maps on distributive lattices. We will employ this theorem heavily in Section~\ref{S:TOIMH}.

\begin{theorem}\label{T:charsoftoi}
Let $L$ be a distributive lattice and $A$ be a nonempty set. Put $n\in\mathbb{N}\setminus\{1\}$. Suppose $T\colon L^n\to A$ is a map. The following are equivalent.
\begin{itemize}
	\item[(1)] $T$ is total orderization invariant,
	\item[(2)] $T$ is symmetric and $T(x_1,x_2,x_3,...,x_n)=T(x_1\wedge x_2, x_1\vee x_2, x_3,...,x_n)$ holds for every $x_1,...,x_n\in L$,
	\item[(3)] 
$T(x_1,...,x_i,....,x_j,...,x_n)=T(x_1,...,x_i\wedge x_j,....,x_i\vee x_j,...,x_n)$ for all $x_1,...,x_n\in L$ and every $i,j\in\{1,...,n\}$, and
	\item[(4)] $T(x_1,...,x_n)=T\Bigl(\M_{1,m}(x_1, ..., x_m), ..., \M_{m,m}(x_1,....,x_m), x_{m+1},..., x_n\Bigr)$ for every $x_1,...,x_n\in L$ and all $m\in\{2,...,n\}$.
\end{itemize}
\end{theorem}
	
\begin{proof}
$(1)\implies(2)$: Suppose $T$ is total orderization invariant. Let $x_1,...,x_n\in L$. By \cite[Proposition~2.9]{Sch-toi}, we know that $T$ is symmetric. The rest follows at once from Lemma~\ref{L:Mk'saretoi}:
\begin{align*}
T(x_1,x_2,...,x_n)&=T\Bigl(\M_1(x_1,x_2,...,x_n),...,\M_n(x_1,x_2,...,x_n)\Bigr)\\
&=T\Bigl(\M_1(x_1\wedge x_2,x_1\vee x_2, x_3,...,x_n),...,\M_n(x_1\wedge x_2, x_1\vee x_2, x_3...,x_n)\Bigr)\\
&=T(x_1\wedge x_2,x_1\vee x_2, x_3,...,x_n).
\end{align*}

$(2)\iff(3)$: Evident.

$(3)\implies(4)$: Suppose $T$ satisfies condition $(3)$ of the theorem. Fix $x_1,...,x_n\in L$. We will show via mathematical induction on $m$ that for every $m\in\{2,...,n\}$, we have
\[
T(x_1,...,x_n)=T\Bigl(\M_{1,m}(x_1, ..., x_m), ..., \M_{m,m}(x_1,....,x_m), x_{m+1},..., x_n\Bigr).
\]

First note that the base step with $m=2$ follows directly from our assumption that $T$ satisfies condition $(3)$:
\begin{align*}
T(x_1, x_2, x_3,..., x_n)&=T(x_1\wedge x_2, x_1\vee x_2, x_3,...,x_n)\\
&=T\Bigl(\M_{1,2}(x_1, x_2), \M_{2,2}(x_1, x_2), x_3, ..., x_n\Bigr).
\end{align*}

(Thus if $n=2$, then (4) holds.) We assume that $n\geq 3$ for the rest of the proof.

For the inductive step, assume that $2\leq m \leq n-1$ and that
\[
T(x_1,...,x_n)=T\Bigl(\M_{1,m}(x_1, ..., x_m), ..., \M_{m,m}(x_1,....,x_m), x_{m+1},..., x_n\Bigr),
\]
written differently, that
\[
T(x_1,...,x_n)=T\Bigl(\M_{1,m}(\bar{x}_{(m)}), ..., \M_{m,m}(\bar{x}_{(m)}), x_{m+1},..., x_n\Bigr)
\]
holds. Using the inductive hypothesis, the assumption of statement $(3)$ of the theorem, the definition of $k$th medians, Lemma~\ref{L:nton+1}, and Proposition~\ref{P:3ton}(iii), we obtain
\begin{align*}
T&(x_1,...,x_n)=T\Bigl(\M_{1,m}(\bar{x}_{(m)}), \M_{2,m}(\bar{x}_{(m)}),..., \M_{m,m}(\bar{x}_{(m)}), x_{m+1},..., x_n\Bigr)\\
\\
&=T\Bigl(\M_{1,m}(\bar{x}_{(m)})\wedge x_{m+1}, \M_{2,m}(\bar{x}_{(m)}),..., \M_{m,m}(\bar{x}_{(m)}), \M_{1,m}(\bar{x}_{(m)})\vee x_{m+1}, x_{m+2},..., x_n\Bigr)\\
\\
&=T\Bigl(\M_{1,m+1}(\bar{x}_{(m+1)}), \M_{2,m}(\bar{x}_{(m)}),..., \M_{m,m}(\bar{x}_{(m)}), \M_{1,m}(\bar{x}_{(m)})\vee x_{m+1}, x_{m+2},..., x_n\Bigr)\\
\\
&=T\biggl(\M_{1,m+1}(\bar{x}_{(m+1)}), \M_{2,m}(\bar{x}_{(m)})\wedge\Bigl[\M_{1,m}(\bar{x}_{(m)})\vee x_{m+1}\Bigr],..., \M_{m,m}(\bar{x}_{(m)}),\\
&\qquad\qquad\M_{2,m}(\bar{x}_{(m)})\vee\Bigl[\M_{1,m}(\bar{x}_{(m)})\vee x_{m+1}\Bigr],x_{m+2},..., x_n\biggr)\\
\\
&=T\Bigl(\M_{1,m+1}(\bar{x}_{(m+1)}), \M_{2,m+1}(\bar{x}_{(m+1)}), \M_{3,m}(\bar{x}_{(m)}),..., \M_{m,m}(\bar{x}_{(m)}),\\
&\qquad\qquad\M_{2,m}(\bar{x}_{(m)})\vee x_{m+1},x_{m+2},..., x_n\Bigr)\\
\\
&=T\biggl(\M_{1,m+1}(\bar{x}_{(m+1)}), \M_{2,m+1}(\bar{x}_{(m+1)}), \M_{3,m}(\bar{x}_{(m)})\wedge\Bigl[\M_{2,m}(\bar{x}_{(m)})\vee x_{m+1}\Bigr],\\
&\qquad\qquad\M_{4,m}(\bar{x}_{(m)}),..., \M_{m,m}(\bar{x}_{(m)}),\M_{3,m}(\bar{x}_{(m)})\vee \Bigl[\M_{2,m}(\bar{x}_{(m)})\vee x_{m+1}\Bigr],\\
&\qquad\qquad x_{m+2},..., x_n\biggr)\\
\\
&=T\Bigl(\M_{1,m+1}(\bar{x}_{(m+1)}), \M_{2,m+1}(\bar{x}_{(m+1)}), \M_{3,m+1}(\bar{x}_{(m+1)}), \M_{4,m}(\bar{x}_{(m)}),...,\\ &\qquad\qquad\M_{m,m}(\bar{x}_{(m)}),\M_{3,m}(\bar{x}_{(m)})\vee x_{m+1}, x_{m+2},..., x_n\Bigr).
\end{align*}
Repeating this procedure, we get
\begin{align*}
T&(x_1,...,x_n)=T\Bigl(\M_{1,m+1}(\bar{x}_{(m+1)}),..., \M_{m-1,m+1}(\bar{x}_{(m+1)}),\M_{m,m}(\bar{x}_{(m)}),\\
&\qquad\qquad\M_{m-1,m}(\bar{x}_{(m)})\vee x_{m+1}, x_{m+2},..., x_n\Bigr)\\
\\
&=T\biggl(\M_{1,m+1}(\bar{x}_{(m+1)}),..., \M_{m-1,m+1}(\bar{x}_{(m+1)}),\M_{m,m}(\bar{x}_{(m)})\wedge\Bigl[\M_{m-1,m}(\bar{x}_{(m)})\vee x_{m+1}\Bigr],\\ &\qquad\qquad\M_{m,m}(\bar{x}_{(m)})\vee\Bigl[\M_{m-1,m}(\bar{x}_{(m)})\vee x_{m+1}\Bigr], x_{m+2},..., x_n\biggr)\\
\\
&=T\Bigl(\M_{1,m+1}(\bar{x}_{(m+1)}),..., \M_{m-1,m+1}(\bar{x}_{(m+1)}), \M_{m,{m+1}}(\bar{x}_{(m+1)})\\
&\qquad\qquad \M_{m,m}(\bar{x}_{(m)})\vee x_{m+1}, x_{m+2},..., x_n\Bigr)\\
\\
&=T\Bigl(\M_{1,m+1}(\bar{x}_{(m+1)}),..., \M_{m-1,m+1}(\bar{x}_{(m+1)}),\\ &\qquad\qquad \M_{m,{m+1}}(\bar{x}_{(m+1)}), \M_{m+1,m+1}(\bar{x}_{(m+1)}), x_{m+2},..., x_n\Bigr)\\
\\
&=T\Bigl(\M_{1,m+1}(x_1, ..., x_{m+1}),..., \M_{m-1,m+1}(x_1, ..., x_{m+1}), \M_{m,{m+1}}(x_1,....,x_{m+1}),\\
&\qquad\qquad \M_{m+1,m+1}(x_1,....,x_{m+1}), x_{m+2},..., x_n\Bigr).
\end{align*}

By the principle of mathematical induction, we obtain that for every $m\in\{2,...,n\}$,
\[
T(x_1,...,x_n)=T\Bigl(\M_{1,m}(x_1, ..., x_m), ..., \M_{m,m}(x_1,....,x_m), x_{m+1},..., x_n\Bigr).
\]

$(4)\implies (1)$: Evident.
\end{proof}	

In light of Theorem~\ref{T:synth}, two corollaries involving bounded orthosymmetric multilinear maps and finite sums of bounded orthogonally additive polynomials in the uniformly complete Archimedean vector lattice setting immediately follow from Theorem~\ref{T:charsoftoi} above. We note that, by \cite[Theorem~3.14 \& Proposition~2.9]{Sch-toi}, the orthosymmetric maps in Corollary~\ref{C:charsoforthosymm} are symmetric. Moreover, the maps
\[
(x_1,...,x_n)\mapsto \sum_{k=1}^n P(x_k)
\]
in Corollary~\ref{C:charsoforthadd'vepolys} are clearly symmetric. Thus the statements of the following two corollaries are shorter than Theorem~\ref{T:charsoftoi}.

\begin{corollary}\label{C:charsoforthosymm}
	Let $E$ be a uniformly complete Archimedean vector lattice, put $n\in\N\setminus\{1\}$, and let $Y$ be a real separated bornological space. Suppose $T\colon E^n\to Y$ is a bounded $n$-linear map. The following are equivalent.
	\begin{itemize}
		\item[(1)] $T$ is orthosymmetric,
			\item[(2)] $T(x_1,x_2,x_3,...,x_n)=T(x_1\wedge x_2, x_1\vee x_2, x_3,...,x_n)$ holds for every $x_1,...,x_n\in E$, and
		\item[(3)] $T(x_1,...,x_n)=T\Bigl(\M_{1,m}(x_1, ..., x_m), ..., \M_{m,m}(x_1,....,x_m), x_{m+1},..., x_n\Bigr)$ for all\\
	$x_1,...,x_n\in E$ each $m\in\{2,...,n\}$.
	\end{itemize}
\end{corollary}

\begin{corollary}\label{C:charsoforthadd'vepolys}
	Let $E$ be a uniformly complete Archimedean vector lattice, and let $Y$ be a real separated bornological space. Fix $n\in\N\setminus\{1\}$. Suppose $P\colon E^n\to Y$ is a bounded $n$-homogeneous polynomial. The following are equivalent.
	\begin{itemize}
		\item[(1)] $P$ is orthogonally additive,
		\item[(2)] 
		$P(x_1)+P(x_2)=P(x_1\wedge x_2)+P(x_1\vee x_2)$ for each $x_1,x_2\in E$, and
	\item[(3)] $\sum_{k=1}^{m}P(x_k)=\sum_{k=1}^{m}P\Bigl(\M_{k,m}(x_1, ..., x_m)\Bigr)$ for every $x_1,...,x_n\in E$ and all $m\in\{2,...,n\}$.
	\end{itemize}
\end{corollary}

In light of \cite[Theorem~3.5]{Sch-toi}, an additional corollary is also a direct consequence of Theorem~\ref{T:charsoftoi} of this paper. The reader is referred to \cite{BusSch} for more information on $h$-completions of Archimedean vector lattices relative to a continuous and positively homogeneous function $h\colon\R^n\to\R$. We note here that \cite[Theorem~3.5]{Sch-toi} states such an $h$, when defined on an Archimedean vector lattice via functional calculus, is total orderization invariant if and only if it is symmetric.
	
\begin{corollary}
Let $n\in\N\setminus\{1\}$, and assume $h\colon\mathbb{R}^n\to\mathbb{R}$ is continuous and positively homogeneous. Suppose $E$ is an $h$-complete Archimedean real vector lattice. The following are equivalent.
	\begin{itemize}
	\item[(1)] $h$ is symmetric,
	\item[(2)] 
	$h(x_1,...,x_i,....,x_j,...,x_n)=h(x_1,...,x_i\wedge x_j,....,x_i\vee x_j,...,x_n)$ for all $x_1,...,x_n\in E$ and every $i,j\in\{1,...,n\}$, and
		\item[(3)] $h(x_1,...,x_n)=h\Bigl(\M_{1,m}(x_1, ..., x_m), ..., \M_{m,m}(x_1,....,x_m), x_{m+1},..., x_n\Bigr)$ for all\\
	$x_1,...,x_n\in E$ each $m\in\{2,...,n\}$.
\end{itemize}
\end{corollary}

	\section{Applications to Multi-Homomorphisms}\label{S:TOIMH}

In this section we employ Theorem~\ref{T:charsoftoi} to achieve various results for total orderization invariant lattice multi-homomorphisms. Our first of these results, as mentioned in the introduction, generalizes \cite[Lemma~2.1]{BoBus}, which proves the equivalence between orthosymmetric vector lattice multimorphisms and symmetric vector lattice multimorphisms. 

\begin{theorem}\label{T:charoftoimultihoms}
Let $L$ and $M$ be distributive lattices, put $n\in\mathbb{N}\setminus\{1\}$, and suppose that $T\colon L^n\to M$ is a lattice $n$-homomorphism. Then $T$ is total orderization invariant if and only if $T$ is symmetric.
\end{theorem}

\begin{proof}
It follows from \cite[Proposition~2.9]{Sch-toi} that if $T$ is total orderization invariant map, then $T$ is symmetric.

Assume $T$ is symmetric, and let $x_1,...,x_n\in L$. We have
\begin{align*}
	T&(x_1\wedge x_2, x_1\vee x_2, x_3,...,x_n)=T(x_1\wedge x_2, x_1, x_3,...,x_n)\vee T(x_1\wedge x_2, x_2, x_3,..., x_n)\\
	&=\Bigl[T(x_1, x_1, x_3, ..., x_n)\wedge T(x_2, x_1, x_3, ..., x_n)\Bigr]\vee\Bigl[T(x_1, x_2, x_3, ..., x_n)\wedge T(x_2, x_2, x_3, ..., x_n)\Bigr]\\
&=\Bigl[T(x_1, x_1, x_3, ..., x_n)\wedge T(x_1, x_2, x_3, ..., x_n)\Bigr]\vee\Bigl[T(x_1, x_2, x_3, ..., x_n)\wedge T(x_2, x_2, x_3, ..., x_n)\Bigr]\\
&=T(x_1,...,x_n)\wedge\Bigl[T(x_1, x_1, x_3,..., x_n)\vee T(x_2, x_2, x_3, ..., x_n)\Bigr]\\
&\leq T(x_1,...,x_n),
\end{align*}
and
\begin{align*}
	T&(x_1\wedge x_2, x_1\vee x_2, x_3,...,x_n)=T(x_1, x_1\vee x_2, x_3,...,x_n)\wedge T(x_2, x_1\vee x_2, x_3,..., x_n)\\
	&=\Bigl[T(x_1, x_1, x_3, ..., x_n)\vee T(x_1, x_2, x_3, ..., x_n)\Bigr]\wedge\Bigl[T(x_2, x_1, x_3, ..., x_n)\vee T(x_2, x_2, x_3, ..., x_n)\Bigr]\\
	&=\Bigl[T(x_1, x_1, x_3, ..., x_n)\vee T(x_1, x_2, x_3, ..., x_n)\Bigr]\wedge\Bigl[T(x_1, x_2, x_3, ..., x_n)\vee T(x_2, x_2, x_3, ..., x_n)\Bigr]\\
	&=T(x_1,...,x_n)\vee\Bigl[T(x_1, x_1, x_3,..., x_n)\wedge T(x_2, x_2, x_3, ..., x_n)\Bigr]\\
	&\geq T(x_1,...,x_n).
\end{align*}
Therefore, we obtain
\[
T(x_1\wedge x_2, x_1\vee x_2, x_3,...,x_n)=T(x_1,...,x_n).
\]
It follows from Theorem~\ref{T:charsoftoi} that $T$ is total orderization invariant.
\end{proof}

The proof of Theorem~\ref{T:charoftoimultihoms} above actually reveals an additional property of symmetric lattice multi-homomorphisms.

\begin{corollary}
If $L$ and $M$ are distributive lattices, $n\in\mathbb{N}\setminus\{1\}$, and $T\colon L^n\to M$ is a symmetric $n$-homomorphism, then for every $x_1,...,x_n\in L$ we have
\begin{align*}
T(x_1,x_1, x_3,..., x_n)\wedge T(x_2, x_2, x_3, ..., x_n)&\leq T(x_1,...,x_n)\\
&\leq T(x_1,x_1, x_3,..., x_n)\vee T(x_2, x_2, x_3, ..., x_n).
\end{align*}
\end{corollary}

\begin{proof}
Suppose $L$ and $M$ are distributive lattices, $n\in\mathbb{N}\setminus\{1\}$, and $T$ is a symmetric $n$-homomorphism. From the proof of Theorem~\ref{T:charoftoimultihoms}, we have
\[
T(x_1\wedge x_2, x_1\vee x_2, x_3,...,x_n)=T(x_1,...,x_n)\wedge\Bigl[T(x_1, x_1, x_3,..., x_n)\vee T(x_2, x_2, x_3, ..., x_n)\Bigr]
\]
and 
\[
T(x_1\wedge x_2, x_1\vee x_2, x_3,...,x_n)=T(x_1,...,x_n)\vee \Bigl[T(x_1, x_1, x_3,..., x_n)\wedge T(x_2, x_2, x_3, ..., x_n)\Bigr].
\]
By Theorem~\ref{T:charoftoimultihoms}, $T$ is total orderization invariant. Using Theorem~\ref{T:charsoftoi}, the above two identities simplify to
\[
T(x_1,...,x_n)=T(x_1,...,x_n)\wedge\Bigl[T(x_1, x_1, x_3,..., x_n)\vee T(x_2, x_2,..., x_n)\Bigr]
\]
and 
\[
T(x_1,...,x_n)=T(x_1,...,x_n)\vee \Bigl[T(x_1, x_1, x_3,..., x_n)\wedge T(x_2, x_2,..., x_n)\Bigr],
\]
respectively. The conclusion of the corollary immediately follows.
\end{proof}

We next turn to the topic of diagonal maps of symmetric lattice multi-homomorphisms. The following notation will be convenient.

\begin{notation}
Given nonempty sets $A$ and $B$, $m,n\in\mathbb{N}$ with $m\leq n$,  $k_1,...,k_m\in\mathbb{N}\cup\{0\}$ with $k_1+\cdots+k_m=n$, a symmetric map $T\colon A^n\to B$, and $x_{k_1},...,x_{k_m}\in A$, we write
\[
T(x_{k_1}^{k_1}, x_{k_2}^{k_2},...,x_{k_m}^{k_m}):=T(\underbrace{x_{k_1}, x_{k_1}, ..., x_{k_1}}_{k_1\ \text{times}}, \underbrace{x_{k_2}, x_{k_2}, ..., x_{k_2}}_{k_2\ \text{times}}, ..., \underbrace{x_{k_m}, x_{k_m}, ..., x_{k_m}}_{k_m\ \text{times}}).
\]
In particular, for $x\in A$ we set
\[
T(x^n):=T(\underbrace{x, x, ..., x}_{n\ \text{times}}).
\]
\end{notation}

Diagonal maps of symmetric lattice multi-homomorphisms are clearly analogous to homogeneous polynomials on vector spaces.

\begin{definition}
Let $L$ and $M$ be distributive lattices, and let $n\in\mathbb{N}$. Given a symmetric lattice $n$-homomorphism $T\colon L^n\to M$, we define its \textit{diagonal map} $P_T$ to be the map $P_T\colon L\to M$ defined by
\[
P_T(x):=T(x^n)\quad (x\in L).
\]
\end{definition}

Given a symmetric lattice multi-homomorphism $T$, we will use $P_T$ to denote its diagonal map throughout the remainder of this manuscript.

If $L$ and $M$ are distributive lattices, $n\in\mathbb{N}\setminus\{1\}$, $T\colon L^n\to M$ is a symmetric lattice $n$-homomorphism, and $x_1,...,x_n\in L$, then
\begin{align*}
	P_T\left(\bigvee_{i=1}^n x_i\right) &= \underset{\begin{subarray}{c}
			k_1,...,k_m\in [n]\cup\{0\} \\
			k_1+\cdots +k_m = n
	\end{subarray}}{\bigvee} T(x_1^{k_1}, x_2^{k_2},..., x_n^{k_n})\\
	&=\bigvee_{i=1}^n P_T(x_i)\ \vee \ \underset{\begin{subarray}{c}
			k_1,...,k_n\in [n-1]\cup\{0\} \\
			k_1+\cdots +k_n = n
	\end{subarray}}{\bigvee} T(x_1^{k_1}, x_2^{k_2},..., x_n^{k_n}),
\end{align*}
and similarly,
\begin{align*}
	P_T\left(\bigwedge_{i=1}^n x_i\right)&=\bigwedge_{i=1}^n P_T(x_i)\ \wedge \ \underset{\begin{subarray}{c}
			k_1,...,k_n\in [n-1]\cup\{0\} \\
			k_1+\cdots +k_n = n
	\end{subarray}}{\bigwedge} T(x_1^{k_1}, x_2^{k_2},..., x_n^{k_n}).
\end{align*}

Thus $P_T$ is a lattice homomorphism if and only if
\[
\bigwedge_{i=1}^n P_T(x_i)\leq T(x_1^{k_1}, x_2^{k_2},..., x_n^{k_n})\leq  \bigvee_{i=1}^n P_T(x_i)
\]
for all $x_1,...,x_n\in L$ and every $k_1,...,k_n\in [n-1]\cup\{0\}$ with $k_1+\cdots+k_n=n$.

In particular, if $P_T$ is a lattice homomorphism, then
\[
\bigwedge_{i=1}^n P_T(x_i)\leq T(x_1, x_2,..., x_n)\leq  \bigvee_{i=1}^n P_T(x_i)
\]
holds for every $x_1,...,x_n\in L$.

On the other hand, assume that
\[
\bigwedge_{i=1}^n P_T(x_i)\leq T(x_1, x_2,..., x_n)\leq  \bigvee_{i=1}^n P_T(x_i)\qquad (x_1,...,x_n\in L).
\]
This supposition implies that, for all $x_1,...,x_n\in L$ and every $k_1,...,k_n\in [n-1]\cup\{0\}$ for which $k_1+\cdots+k_n=n$, we have
\[
\bigwedge_{\{i \in [n] :\ 
	k_i \neq 0\}} P_T(x_i) \leq T(x_1^{k_1}, x_2^{k_2},..., x_n^{k_n})\leq  \bigvee_{\{i \in [n] :\ 
	k_i \neq 0\}} P_T(x_i).
\]
We thus obtain
\[
\bigwedge_{i=1}^n P_T(x_i)\leq \bigwedge_{\{i \in [n] :\ 
	k_i \neq 0\}} P_T(x_i) \leq T(x_1^{k_1}, x_2^{k_2},..., x_n^{k_n})\leq  \bigvee_{\{i \in [n] :\ 
	k_i \neq 0\}} P_T(x_i) \leq \bigvee_{i=1}^n P_T(x_i),
\]
so that
\[
\bigwedge_{i=1}^n P_T(x_i)\leq T(x_1^{k_1}, x_2^{k_2},..., x_n^{k_n})\leq  \bigvee_{i=1}^n P_T(x_i)
\]
holds for each $x_1,...,x_n\in L$ and all $k_1,...,k_n\in [n-1]\cup\{0\}$ such that $k_1+\cdots+k_n=n$.

The following link between $T$ and $P_T$ instantly follows. As mentioned in the introduction, Lemma~\ref{L:kusraev} generalizes the equivalence of items $(iii)$ and $(iv)$ in Theorem~\ref{T:Kus} to a statement involving several variables and a more general setting.

\begin{lemma}\label{L:kusraev}	
	Let $n\in\mathbb{N}\setminus\{1\}$ and $L$ and $M$ be distributive lattices. If $T\colon L^n\to M$ is a symmetric lattice $n$-homomorphism, then $P_T$ is a lattice homomorphism if and only if for all $x_1,..., x_n\in L$ we have
	\[
	\bigwedge_{i=1}^n P_T(x_i)\leq T(x_1,...,x_n)\leq \bigvee_{i=1}^n P_T(x_i).
	\]
\end{lemma}

The following corollary of Lemma~\ref{L:kusraev} will exploited in the proof of Theorem~\ref{T:diags}.

\begin{corollary}\label{C:wCSI}
Let $L$ and $M$ be distributive lattices. Put $n\in\mathbb{N}\setminus\{1\}$, and assume $T\colon L^n\to M$ is a symmetric lattice $n$-homomorphism. Then $P_T$ is a lattice homomorphism if and only if for all $x,y\in L$ and every $k\in[n-1]$ we have
\[
P_T(x)\wedge P_T(y)\leq T(x^k, y^{n-k})\leq P_T(x) \vee P_T(y).
\] 
\end{corollary}

\begin{proof}
If $P_T$ is a lattice homomorphism, then
\[
P_T(x)\wedge P_T(y)\leq T(x^k, y^{n-k})\leq P_T(x) \vee P_T(y)
\]
for all $x,y\in L$ and every $k\in[n-1]$ by Lemma~\ref{L:kusraev}.

Conversely, if $x,y\in L$ and
\[
P_T(x)\wedge P_T(y)\leq T(x^k, y^{n-k})\leq P_T(x) \vee P_T(y)
\]
holds for every $k\in[n-1]$ then
\begin{align*}
P_T(x)\vee P_T(y)&=P_T(x)\vee P_T(y)\vee \bigvee_{k=1}^{n-1}T(x^k, y^{n-k})\\
&=\bigvee_{k=0}^{n}T(x^k, y^{n-k})\\
&=P_T(x\vee y),
\end{align*}
and similarly,
\[
P_T(x)\wedge P_T(y)=P_T(x\wedge y).
\]
Thus $P_T$ is a lattice homomorphism.
\end{proof}

The following lemma is needed for the proof of Theorem~\ref{T:diags}, which states that diagonal maps of symmetric lattice multi-homomorphisms are lattice homomorphisms.

\begin{lemma}\label{L:alg}
Suppose $L$ and $M$ are distributive lattices, $n\in\mathbb{N}\setminus\{1\}$, and $T\colon L^n\to M$ is a symmetric lattice $n$-homomorphism. Let $j, m, p\in\mathbb{N}\cup\{0\}$ satisfy $1\leq m\leq p\leq p+j\leq n-1$.
For each $x,y\in L$ we have
\[
\bigvee_{i=m}^{p}T\Bigl(x^i, y^{n-i-j}, (x\wedge y)^{j}\Bigr)=\bigvee_{i=m-1}^p T\Bigl(x^i, y^{n-i-j-1}, (x\wedge y)^{j+1}\Bigr).
\]
\end{lemma}

\begin{proof}
(1) Let $x,y\in L$. Since $T$ is a symmetric $n$-homomorphism, $T$ is total orderization invariant by Theorem~\ref{T:charoftoimultihoms}. Hence by Theorem~\ref{T:charsoftoi}, we know that
\[
T(x_1,...,x_i,....,x_j,...,x_n)=T(x_1,...,x_i\wedge x_j,....,x_i\vee x_j,...,x_n)
\]
 for each $x_1,...,x_n\in L$ and all $i,j\in\{1,...,n\}$. Using the identity above as well as the assumption that $T$ is a symmetric lattice $n$-homomorphism, we obtain
\begin{align*}
\bigvee_{i=m}^{p}T&\Bigl(x^i, y^{n-i-j}, (x\wedge y)^{j}\Bigr)=\bigvee_{i=m}^{p}T\Bigl(x^{i-1}, y^{n-i-j-1}, (x\wedge y)^{j+1}, x\vee y\Bigr)\\
&=\bigvee_{i=m}^{p}\left\{T\Bigl(x^{i}, y^{n-i-j-1}, (x\wedge y)^{j+1}\Bigr) \vee T\Bigl(x^{i-1}, y^{n-i-j}, (x\wedge y)^{j+1}\Bigr)\right\}\\
&=\bigvee_{i=m}^{p}T\Bigl(x^{i}, y^{n-i-j-1}, (x\wedge y)^{j+1}\Bigr) \vee \bigvee_{i=m}^p T\Bigl(x^{i-1}, y^{n-i-j}, (x\wedge y)^{j+1}\Bigr)\\\
&=\bigvee_{i=m}^{p}T\Bigl(x^{i}, y^{n-i-j-1}, (x\wedge y)^{j+1}\Bigr) \vee \bigvee_{i=m-1}^{p-1} T\Bigl(x^i, y^{n-i-j-1}, (x\wedge y)^{j+1}\Bigr)\\\
&=\bigvee_{i=m-1}^p T\Bigl(x^i, y^{n-i-j-1}, (x\wedge y)^{j+1}\Bigr).
\end{align*}
\end{proof}

We are now ready to prove that diagonal maps of symmetric lattice multi-homomorphisms between distributive lattices are lattice homomorphisms.

\begin{theorem}\label{T:diags}
Let $L$ and $M$ be distributive lattices, put $n\in\N\setminus\{1\}$, and let $T\colon L^n\to M$ be a lattice $n$-homomorphism. If $T$ is symmetric, then $P_T$ is a lattice homomorphism.
\end{theorem}

\begin{proof}
Suppose $T$ is symmetric. We know from Corollary~\ref{C:wCSI} that $P_T$ is a lattice homomorphism if and only if for all $x,y\in L$ and each $k\in[n-1]$, we have
\[
P_T(x)\wedge P_T(y)\leq T(x^k, y^{n-k})\leq P_T(x) \vee P_T(y).
\]
We will show that
\begin{equation}\tag{$\ast$}\label{eq: ast}
T(x^k, y^{n-k})\leq P_T(x) \vee P_T(y)\qquad \Bigl(x,y\in L,\ k\in[n-1]\Bigr)
\end{equation}
holds, noting that the proof of
\[
P_T(x)\wedge P_T(y)\leq T(x^k, y^{n-k})\qquad \Bigl(x,y\in L,\ k\in[n-1]\Bigr)
\]
is nearly identical.

To this end, let $x,y\in L$ and $k\in[n-1]$. Without loss of generality, suppose $k\leq n-k$.

As a symmetric lattice $n$-homomorphism, $T$ is total orderization invariant by Theorem~\ref{T:charoftoimultihoms}, and thus
\[
T(x_1,...,x_i,....,x_j,...,x_n)=T(x_1,...,x_i\wedge x_j,....,x_i\vee x_j,...,x_n)
\]
holds for each $x_1,...,x_n\in L$ and all $i,j\in\{1,...,n\}$ by Theorem~\ref{T:charsoftoi}. Using the previous identity in the first equality below, the assumption that $T$ is a symmetric lattice multi-homomorphism in the second identity below, and then repeatedly employing Lemma~\ref{L:alg}, we get
\begin{align*}
T(x^k, y^{n-k})&=T(x^{k-1}, y^{n-k-1}, x\vee y, x\wedge y)\\
&=\bigvee_{i=k-1}^k T(x^{i},  y^{n-i-1}, x\wedge y)\\
&=\bigvee_{i=k-2}^k T\Bigl(x^{i},  y^{n-i-2}, (x\wedge y)^2\Bigr)\\
&=\bigvee_{i=k-3}^k T\Bigl(x^{i},  y^{n-i-3}, (x\wedge y)^3\Bigr)\\
&\hspace{.195cm}\vdots\\
&=\bigvee_{i=0}^k T\Bigl(x^{i},  y^{n-i-k}, (x\wedge y)^{k}\Bigr).
\end{align*}

At this point, we see that the first term in the supremum does not have $x$ as an argument:

\[
\bigvee_{i=0}^k T\Bigl(x^{i},  y^{n-i-k}, (x\wedge y)^{k}\Bigr)=T\Bigl(y^{n-k}, (x\wedge y)^k\Bigr) \vee \bigvee_{i=1}^k T\Bigl(x^{i},  y^{n-i-k}, (x\wedge y)^{k}\Bigr).
\]

Using Lemma~\ref{L:alg} over and over again, we obtain
\begin{align*}
T&\Bigl(y^{n-k}, (x\wedge y)^k\Bigr) \vee \bigvee_{i=1}^k T\Bigl(x^{i},  y^{n-i-k}, (x\wedge y)^{k}\Bigr)\\
\\
&=T\Bigl(y^{n-k}, (x\wedge y)^k\Bigr) \vee \bigvee_{i=0}^k T\Bigl(x^{i},  y^{n-i-(k+1)}, (x\wedge y)^{k+1}\Bigr)\\
\\
&=T\Bigl(y^{n-k}, (x\wedge y)^k\Bigr) \vee T\Bigl(y^{n-(k+1)}, (x\wedge y)^{k+1}\Bigr) \vee \bigvee_{i=1}^k T\Bigl(x^{i},  y^{n-i-(k+1)}, (x\wedge y)^{k+1}\Bigr)\\
&\hspace{.195cm}\vdots\\
&=T\Bigl(y^{n-k}, (x\wedge y)^k\Bigr) \vee \cdots \vee  T\Bigl(y^{n-(k+n-2k-1)}, (x\wedge y)^{k+n-2k-1}\Bigr)\\
&\qquad\qquad \vee \bigvee_{i=0}^k T\Bigl(x^{i},  y^{n-i-(k+n-2k)}, (x\wedge y)^{k+n-2k}\Bigr)\\
\\
&=T\Bigl(y^{n-k}, (x\wedge y)^k\Bigr) \vee \cdots \vee  T\Bigl(y^{k+1}, (x\wedge y)^{n-(k+1)}\Bigr) \vee \bigvee_{i=0}^k T\Bigl(x^{i},  y^{k-i}, (x\wedge y)^{n-k}\Bigr).
\end{align*}

If $k=1$, we have
\begin{align*}
T&\Bigl(y^{n-k}, (x\wedge y)^k\Bigr) \vee \cdots \vee  T\Bigl(y^{k+1}, (x\wedge y)^{n-(k+1)}\Bigr) \vee \bigvee_{i=0}^k T\Bigl(x^{i},  y^{k-i}, (x\wedge y)^{n-k}\Bigr)\\
\\
&=T\Bigl(y^{n-1}, x\wedge y\Bigr) \vee \cdots \vee  T\Bigl(y^{2}, (x\wedge y)^{n-2}\Bigr)\\
&\qquad\qquad \vee T\Bigl(y, (x\wedge y)^{n-1}\Bigr) \vee T\Bigl(T(x, (x\wedge y)^{n-1})\Bigr)\\
\\
&\leq T(y^n)\vee \cdots \vee T(y^n) \vee T(y^n) \vee T(x^n)\\
\\
&=P_T(x)\vee P_T(y),
\end{align*}
using the isotonicity of $T$. Hence \eqref{eq: ast} holds for $k=1$.

Suppose next that $k\geq 2$. As with the case that $k=1$, we now see that from this point on the first term in the supremum
\[
\bigvee_{i=0}^k T\Bigl(x^{i},  y^{k-i}, (x\wedge y)^{n-k}\Bigr)
\]
does not have $x$ as an argument, and the last term in this supremum does not have $y$ as an argument:

\begin{align*}
T&\Bigl(y^{n-k}, (x\wedge y)^k\Bigr) \vee \cdots \vee  T\Bigl(y^{k+1}, (x\wedge y)^{n-(k+1)}\Bigr) \vee \bigvee_{i=0}^k T\Bigl(x^{i},  y^{k-i}, (x\wedge y)^{n-k}\Bigr)\\
\\
&=T\Bigl(y^{n-k}, (x\wedge y)^k\Bigr) \vee \cdots \vee  T\Bigl(y^{k+1}, (x\wedge y)^{n-(k+1)}\Bigr) \vee T\Bigl(y^k, (x\wedge y)^{n-k}\Bigr)\\
&\qquad\qquad\vee T\Bigl(x^k, (x\wedge y)^{n-k}\Bigr) \vee \bigvee_{i=1}^{k-1} T\Bigl(x^{i},  y^{k-i}, (x\wedge y)^{n-k}\Bigr).
\end{align*}

Employing Lemma~\ref{L:alg} repeatedly and taking out terms with no $x$ and no $y$ for argument, we get
\begin{align*}
T&\Bigl(y^{n-k}, (x\wedge y)^k\Bigr) \vee \cdots \vee  T\Bigl(y^{k+1}, (x\wedge y)^{n-(k+1)}\Bigr) \vee T\Bigl(y^k, (x\wedge y)^{n-k}\Bigr)\\
&\qquad\qquad \vee T\Bigl(x^k, (x\wedge y)^{n-k}\Bigr) \vee \bigvee_{i=1}^{k-1} T\Bigl(x^{i},  y^{k-i}, (x\wedge y)^{n-k}\Bigr)\\
\\
&=T\Bigl(y^{n-k}, (x\wedge y)^k\Bigr) \vee \cdots \vee  T\Bigl(y^{k+1}, (x\wedge y)^{n-(k+1)}\Bigr) \vee T\Bigl(y^k, (x\wedge y)^{n-k}\Bigr)\\
&\qquad\qquad \vee T\Bigl(x^k, (x\wedge y)^{n-k}\Bigr)  \vee \bigvee_{i=0}^{k-1} T\Bigl(x^{i},  y^{k-1-i}, (x\wedge y)^{n-(k-1)}\Bigr)\\
\\
&=T\Bigl(y^{n-k}, (x\wedge y)^k\Bigr) \vee \cdots \vee  T\Bigl(y^{k+1}, (x\wedge y)^{n-(k+1)}\Bigr) \vee T\Bigl(y^k, (x\wedge y)^{n-k}\Bigr)\\
&\qquad\qquad \vee T\Bigl(y^{k-1}, (x\wedge y)^{n-(k-1)}\Bigr) \vee T\Bigl(x^k, (x\wedge y)^{n-k}\Bigr) \vee T\Bigl(x^{k-1}, (x\wedge y)^{n-(k-1)}\Bigr)\\
&\qquad\qquad \vee \bigvee_{i=1}^{k-2} T\Bigl(x^{i},  y^{k-1-i}, (x\wedge y)^{n-(k-1)}\Bigr)\\
&\hspace{.195cm}\vdots\\
&=T\Bigl(y^{n-k}, (x\wedge y)^k\Bigr) \vee \cdots \vee   T\Bigl(y^{2}, (x\wedge y)^{n-2}\Bigr) \vee T\Bigl(x^k, (x\wedge y)^{n-k}\Bigr)\\
&\qquad\qquad \vee \cdots \vee T\Bigl(x^{2}, (x\wedge y)^{n-2}\Bigr) \vee \bigvee_{i=1}^{1} T\Bigl(x^{i},  y^{2-i}, (x\wedge y)^{n-2}\Bigr)\\
\\
&=T\Bigl(y^{n-k}, (x\wedge y)^k\Bigr) \vee \cdots \vee   T\Bigl(y^{2}, (x\wedge y)^{n-2}\Bigr) \vee T\Bigl(x^k, (x\wedge y)^{n-k}\Bigr)\\
&\qquad\qquad \vee \cdots \vee T\Bigl(x^{2}, (x\wedge y)^{n-2}\Bigr) \vee T\Bigl(x,  y, (x\wedge y)^{n-2}\Bigr).
\end{align*}

Next we obtain

\begin{align*}
T&\Bigl(y^{n-k}, (x\wedge y)^k\Bigr) \vee \cdots \vee   T\Bigl(y^{2}, (x\wedge y)^{n-2}\Bigr) \vee T\Bigl(x^k, (x\wedge y)^{n-k}\Bigr)\\
&\qquad\qquad \vee \cdots \vee T\Bigl(x^{2}, (x\wedge y)^{n-2}\Bigr) \vee T\Bigl(x,  y, (x\wedge y)^{n-2}\Bigr)\\
\\
&=T\Bigl(y^{n-k}, (x\wedge y)^k\Bigr) \vee \cdots \vee   T\Bigl(y^{2}, (x\wedge y)^{n-2}\Bigr) \vee T\Bigl(x^k, (x\wedge y)^{n-k}\Bigr)\\
&\qquad\qquad \vee \cdots \vee T\Bigl(x^{2}, (x\wedge y)^{n-2}\Bigr) \vee T\Bigl(x \vee y, (x\wedge y)^{n-1}\Bigr)\\
\\
&=T\Bigl(y^{n-k}, (x\wedge y)^k\Bigr) \vee \cdots \vee   T\Bigl(y^{2}, (x\wedge y)^{n-2}\Bigr) \vee T\Bigl(x^k, (x\wedge y)^{n-k}\Bigr)\\
&\qquad\qquad \vee \cdots \vee T\Bigl(x^{2}, (x\wedge y)^{n-2}\Bigr) \vee T\Bigl(x, (x\wedge y)^{n-1}\Bigr) \vee T\Bigl(y, (x\wedge y)^{n-1}\Bigr).
\end{align*}

Finally, using the isotonicity of $T$, we have
\begin{align*}
T&\Bigl(y^{n-k}, (x\wedge y)^k\Bigr) \vee \cdots \vee   T\Bigl(y^{2}, (x\wedge y)^{n-2}\Bigr) \vee T\Bigl(x^k, (x\wedge y)^{n-k}\Bigr) \vee \cdots \vee T\Bigl(x^{2}, (x\wedge y)^{n-2}\Bigr)\\
&\qquad\qquad \vee T\Bigl(x, (x\wedge y)^{n-1}\Bigr) \vee T\Bigl(y, (x\wedge y)^{n-1}\Bigr)\\
\\
&\leq T(y^n) \vee \cdots \vee T(y^n) \vee T(x^n) \vee \cdots \vee T(x^n) \vee T(x^n) \vee T(y^n)\\
\\
&=P_T(x)\vee P_T(y).
\end{align*}

We thus obtain \eqref{eq: ast}, and the proof is complete.
\end{proof}

As a final note, we in particular have the following result for homogeneous polynomials on vector lattices.

\begin{corollary}
Let $E$ and $F$ be vector lattices, put $n\in\mathbb{N}\setminus\{1\}$, and let $P_T\colon E\to F$ be an $n$-homogeneous polynomial generated by a  symmetric $n$-linear map $T\colon E^n\to F$. If $T$ is a vector lattice $n$-morphism, then $P_T$ is a lattice homomorphism.
\end{corollary}

\end{document}